\documentclass[12pt, english]{article}

\usepackage{lipsum}
\usepackage{etoolbox}
\RequirePackage{ifthen}

\usepackage{amsmath,amssymb,amsfonts}
\usepackage{mathrsfs}
\usepackage{enumitem}

\input{xy}
\xyoption{all}
\RequirePackage[top=4cm,bottom=3.5cm,inner=3.5cm,outer=3.5cm,
headsep=1cm,footskip=1cm]{geometry}

\RequirePackage{xr-hyper}
\RequirePackage{hyperref}

\hypersetup{bookmarksopen=true}%

\ifthenelse{\boolean{pdf}}{%
	\hypersetup{colorlinks=true}}{%
	\hypersetup{colorlinks=false}}

\hypersetup{linkcolor=blue}%
\hypersetup{urlcolor=cyan}%
\hypersetup{citecolor=cyan}

\makeatletter
\let\@afterindenttrue\@afterindentfalse
\patchcmd\@thm
{\let\thm@indent\indent}{\let\thm@indent\noindent}%
{}{}
\makeatother

\newtheorem{thm}{Theorem}[section]
\newtheorem{lemma}[thm]{Lemma}
\newtheorem{corollary}[thm]{Corollary}

\newtheorem{myexamp}{Example}[section]
\newtheorem{definition}[myexamp]{Definition}
\newtheorem{remark}[thm]{Remark}

\newcommand{\qed}{\hfill$\square$}

\begin{document}
	
\begin{center}
	\Large{\textbf{Equality of elementary symplectic group and symplectic group}}
\end{center}
	
\begin{center}
	Ruddarraju Amrutha\footnote{ This article is part of the doctoral thesis of the first named author}, Pratyusha Chattopadhyay
\end{center}
	
	
\medskip

\begin{center} 	
\end{center}

\begin{center}
\end{center}	
\medskip


\begin{center}
	Abstract
\end{center}

V.I. Kopeiko proved that over a euclidean ring, the symplectic group defined with respect to the standard skew-symmetric matrix is same as the elementary symplectic group.  Here we generalise the result of Kopeiko for a symplectic group defined with respect to any invertible skew-symmetric matrix of Pfaffian one.

\medskip


\section{Introduction}
\label{section:1}
	
In \cite{Lam}, T.Y. Lam proved that over a euclidean ring, any matrix of determinant $1$ can be written as a product of elementary linear matrices. In \cite{Kop}, V.I. Kopeiko proved that over a euclidean domain $R$, the symplectic group $\mathrm{Sp}_{2n}(R)$ is same as the elementary symplectic group $\mathrm{ESp}_{2n}(R)$. This means that over a euclidean domain any symplectic matrix can be written as a product of elementary symplectic generators. For a skew-symmetric matrix $\varphi$ of size $2n$, the symplectic group $\mathrm{Sp}_\varphi(R)$ and the elementary symplectic group $\mathrm{ESp}_\varphi(R)$ can be considered as generalizations of $\mathrm{Sp}_{2n}(R)$ and $\mathrm{ESp}_{2n}(R)$ respectively (see Lemma \ref{psi 2n}). In this paper, we prove that over a euclidean domain, the symplectic group with respect to a skew-symmetric matrix of Pfaffian $1$ is same as the corresponding elementary symplectic group (see Theorem \ref{equality}). We also prove a relative version of this result with respect to an ideal of a ring (see Theorem \ref{rel equality}).


\section{Preliminaries}
\label{section:2}
	
Let $R$ be a commutative ring with unity. Let $R^n$ denote the space of column vectors of length $n$ with entries in $R$. The ring of matrices of size $n\times n$ with entries in $R$ is denoted by $\mathrm{M}_{n}(R)$. The identity matrix of size $n\times n$ is denoted by $I_n$ and $e_{ij}$ denotes the $n\times n$ matrix which has $1$ in the $(i,j)$-th position and $0$ everywhere else. The collection of invertible $n\times n$ matrices with entries in $R$ is denoted by $\mathrm{GL}_{n}(R)$. For $\alpha\in\mathrm{M}_{m}(R)$ and $\beta\in\mathrm{M}_{n}(R)$, the matrix {\tiny{$\begin{pmatrix}\alpha&0\\0&\beta\end{pmatrix}$}}, which is an element of $\mathrm{M}_{m+n}(R)$, is denoted by $\alpha\perp\beta$. For a matrix $\gamma$ of size $m \times n$, $\gamma^t$ denotes the transpose of $\gamma$ which is of size $n \times m$.
	
\begin{definition} 
	\rm{An element $v\in R^n$ is said to be \textit{unimodular} if there exists $w\in R^n$ such that their dot product $v^tw=1$. The set of all unimodular elements in $R^n$ is denoted by $\mathrm{Um}_n(R)$.} 
\end{definition} 

\begin{definition} 
	\rm{The \textit{elementary linear group} $\mathrm{E}_{n}(R)$ is a subgroup of $\mathrm{GL}_n(R)$ generated by elements of the form $E_{ij}(a)=I_n+ae_{ij}$, for $a\in R$. For an ideal $I$ of $R$, the subgroup of $\mathrm{E}_{n}(R)$ generated by $E_{ij}(x)$, for $x\in I$ is denoted by $\mathrm{E}_{n}(I)$. The \textit{relative elementary group}, denoted by $\mathrm{E}_{n}(R,I)$, is the normal closure of $\mathrm{E}_{n}(I)$ in $\mathrm{E}_{n}(R)$. In other words, $\mathrm{E}_{n}(R,I)$ is generated by elements of the form $E_{kl}(a) E_{ij}(x) E_{kl}(-a)$, where $a \in R$ and $x \in I$.}
\end{definition} 

\begin{definition} 
		\rm{Let $\psi_n=\sum_{i=1}^{n}(e_{2i-1,2i}-e_{2i,2i-1})$ denote the standard skew-symmetric matrix. The \textit{symplectic group}, denoted by $\mathrm{Sp}_{2n}(R)$, is a subgroup of $\mathrm{GL}_{2n}(R)$ defined as 
	\begin{equation*}
		\mathrm{Sp}_{2n}(R)=\{\alpha\in\mathrm{GL}_{2n}(R)\;\big|\;\alpha^t\psi_n\alpha=\psi_n\}.
	\end{equation*}	

	Given an invertible skew-symmetric matrix $\varphi$ of size $2n$, 
	\begin{equation*}
		\mathrm{Sp}_\varphi(R)=\{\alpha\in\mathrm{GL}_{2n}(R)\;\big|\;\alpha^t\varphi\alpha=\varphi\}.
	\end{equation*}}
\end{definition} 

\begin{definition} 
	\rm{Let $I$ be an ideal of $R$. Then we have the canonical ring homomorphism $f:R\rightarrow R/I$. Using $f$, we can define a ring homomorphism $\tilde{f}:\mathrm{Sp}_\varphi(R)\rightarrow\mathrm{Sp}_{\varphi}(R/I)$ given by $\tilde{f}(a_{ij})=(f(a_{ij}))$. We denote the kernel of this map by $\mathrm{Sp}_{\varphi}(R,I)$.}
\end{definition} 

\begin{definition} 
    \rm{Let $\sigma$ be the permutation of $\{1,2,\ldots,2n\}$ given by $\sigma(2i-1)=2i$ and $\sigma(2i)=2i-1$. For $a\in R$ and $1\leq i\neq j\leq 2n$, define $se_{ij}(a)$ as
    \begin{equation*}
	se_{ij}(a)=\begin{cases} I_n+ae_{ij},&\text{ if } i=\sigma(j)\\
	I_n+ae_{ij}-(-1)^{i+j}ae_{\sigma(j)\sigma(i)},&\text{ if } i\neq\sigma(j). \end{cases}
	\end{equation*}
	
	Note that $se_{ij}(a)\in\mathrm{Sp}_{2n}(R)$. These matrices are called the \textit{elementary symplectic matrices}. The subgroup of $\mathrm{Sp}_{2n}(R)$ generated by the elementary symplectic matrices is called the \textit{elementary symplectic group}, and is denoted by $\mathrm{ESp}_{2n}(R)$. We denote by $\mathrm{ESp}_{2n}(I)$ a subgroup of $\mathrm{ESp}_{2n}(R)$ generated by elements of the form $se_{ij}(x)$, for $x\in I$. The \textit{relative elementary group}, denoted by $\mathrm{ESp}_{2n}(R,I)$, is the normal closure of $\mathrm{ESp}_{2n}(I)$ in $\mathrm{ESp}_{2n}(R)$. In other words, $\mathrm{ESp}_{2n}(R, I)$ generated by elements of the form $se_{kl}(a) se_{ij}(x) se_{kl}(-a)$, where $a \in R$ and $x \in I$.}
\end{definition} 

\begin{remark}
	\rm{For a ring $R$ with $R=2R$ and an ideal $I$ of $R$, the relative elementary symplectic group $\mathrm{ESp}_{2n}(R,I)$ is the smallest normal subgroup of $\mathrm{ESp}_{2n}(R)$ containing $se_{21}(x)$, for all $x\in I$. A similar result for the elementary linear group was proved by W. van der Kallen in \cite{van} (Lemma 2.2). The result for symplectic group can be proved using a similar argument present in Lemma A.1 of \cite{JPAA}.}
\end{remark}

\begin{definition} 
	\rm{Let $\varphi$ be an invertible skew-symmetric matrix of size $2n$ of the form $\begin{pmatrix} 0 & -c^t\\ c & \nu \end{pmatrix}$, and $\varphi^{-1}$ be of the form $\begin{pmatrix} 0 & d^t\\ -d & \mu \end{pmatrix}$, where $c,d\in R^{2n-1}$ and $\nu,\mu\in\mathrm{M}_{2n-1}(R)$.
		
	Given $v\in R^{2n-1}$, consider the matrices $\alpha$ and $\beta$ defined as 
	\begin{equation*}
	\begin{aligned}
		\alpha&:=\alpha_\varphi(v)&:=I_{2n-1}+dv^t\nu\\
		\beta&:=\beta_\varphi(v)&:=I_{2n-1}+\mu vc^t.
	\end{aligned}
	\end{equation*}
		
	L.N. Vaserstein constructed these matrices in Lemma 5.4, \cite{Vas}. Note that $\alpha$ and $\beta$ depend on $\varphi$ and $v$. Also, $\alpha,\beta\in\mathrm{E}_{2n-1}(R)$. This follows by Corollary 1.2 and Lemma 1.3 of \cite{Sus}. One can also see Lemma 9.11 of Chapter 1 of \cite{Lam} for a proof of this result. An interesting observation about these matrices is that $\mathrm{E}_{2n-1}(R)$ is generated by the set  $\{\alpha_\varphi(v),\beta_\varphi(v)\; :\; v\in R^{2n-1}\}$ (Theorem 5.1 of \cite{GSV}). Using these matrices, Vaserstein constructed the following matrices in \cite{Vas}:
	\begin{equation*}
	\begin{aligned}
		C_\varphi(v)&:=\begin{pmatrix} 1&0\\ v&\alpha\end{pmatrix}\\
		R_\varphi(v)&:=\begin{pmatrix} 1&v^t\\ 0&\beta\end{pmatrix}.
	\end{aligned}
	\end{equation*}
		
	Note that $C_\varphi(v)$ and $R_\varphi(v)$ belong to $\mathrm{Sp}_\varphi(R)$. The \textit{elementary symplectic group} $\mathrm{ESp}_{\varphi}(R)$ \textit{with respect to the invertible skew-symmetric matrix} $\varphi$ is a subgroup of $\mathrm{Sp}_\varphi(R)$ generated by $C_\varphi(v)$ and $R_\varphi(v)$, for $v\in R^{2n-1}$. We denote by $\mathrm{ESp}_{\varphi}(I)$ a subgroup of $\mathrm{ESp}_{\varphi}(R)$ generated as a group by the elements $C_\varphi(v)$ and $R_\varphi(v)$, for $v\in I^{2n-1}(\subseteq R^{2n-1})$. The \textit{relative elementary symplectic group} $\mathrm{ESp}_{\varphi}(R,I)$ is the normal closure of $\mathrm{ESp}_{\varphi}(I)$ in $\mathrm{ESp}_{\varphi}(R)$.}
	\end{definition}
	


\section{Results about elementary symplectic group}

In this section, we recall a few results related to the elementary symplectic group. We also obtain Lemma \ref{ESp cap Sp X} which plays a crucial role in the proof of the relative version of the main theorem (Theorem \ref{rel equality}).
	
\begin{lemma}(Lemma 3.6, 3.7, \cite{JOA})
\label{phi phi star}
	Let $\varphi$ and $\varphi^\ast$ be two invertible skew-symmetric matrices such that $\varphi=(1\perp\epsilon)^t\varphi^\ast(1\perp\epsilon)$ for some $\epsilon\in\mathrm{E}_{2n-1}(R)$. Then, we have
	\begin{equation*}
	\begin{aligned}
		\mathrm{Sp}_\varphi(R)&=(1\perp\epsilon)^{-1}\mathrm{Sp}_{\varphi^\ast}(R)(1\perp\epsilon),\\
		\mathrm{ESp}_{\varphi}(R)&=(1\perp\epsilon)^{-1}\mathrm{ESp}_{\varphi^\ast}(R)(1\perp\epsilon).
	\end{aligned}
	\end{equation*}
\end{lemma}
	
A relative version of the above lemma is as follows.

\begin{lemma}(Lemma 3.8, \cite{JOA})
\label{lemma:3.6}
	Let $\varphi$ and $\varphi^\ast$ be two invertible skew-symmetric matrices such that $\varphi=(1\perp\epsilon)^t\varphi^\ast(1\perp\epsilon)$, for some $\epsilon\in\mathrm{E}_{2n-1}(R,I)$. Then,
	\begin{equation*}
		\mathrm{ESp}_{\varphi}(R,I)=(1\perp\epsilon)^{-1}\mathrm{ESp}_{\varphi^\ast}(R,I)(1\perp\epsilon).
	\end{equation*}
\end{lemma}

The following lemma shows that the elementary symplectic group $\mathrm{ESp}_{\varphi}(R)$ with respect to a skew-symmetric matrix $\varphi$ can be considered as a generalization of the elementary symplectic group $\mathrm{ESp}_{2n}(R)$. We include the proof for completeness.

\begin{lemma}(Lemma 3.5, \cite{JOA})
\label{psi 2n}
	Let $R$ be a ring with $R=2R$ and $n\geq 2$. Then,  $\mathrm{ESp}_{\psi_n}(R)=\mathrm{ESp}_{2n}(R)$.
\end{lemma}
	
\begin{proof} 
	$\mathrm{ESp}_{\psi_n}(R)\subseteq\mathrm{ESp}_{2n}(R)$ as for   $v=(a_1,\cdots,a_{2n-1})^t\in R^{2n-1}$, we have 
	\begin{equation*}
		C_{\psi_n}(v)=\prod_{i=2}^{2n}se_{i1}(a_{i-1}) \text{ and }
		R_{\psi_n}(v)=\prod_{i=2}^{2n}se_{1i}(a_{i-1}).
	\end{equation*}
		
	For integers $i,j$ with $i\neq j,\sigma(j)$ and for $a,b\in R$, we have 
	\begin{equation*}
	\begin{aligned}
		\lbrack  se_{i \sigma(i)}(a),se_{\sigma(i) j}(b) \rbrack &=se_{ij}(ab)se_{\sigma(j)j}((-1)^{i+j}ab^2),\\
		[se_{ik}(a),se_{kj}(b)]&=se_{ij}(ab), \text{ if  } k\neq\sigma(i),\sigma(j),\\
		[se_{ik}(a),se_{k\sigma(i)}(b)]&=se_{i\sigma(i)}(2ab), \text{ if } k\neq i,\sigma(i).
	\end{aligned}
	\end{equation*}
		
	Using these identities, $se_{ij}(a)$, for $i,j\neq 1$, can be written as a product of elements of the form $se_{1i}(x)$ and $se_{j1}(y)$, for $x,y\in R$. Also, $se_{1i}(a), se_{j1}(b)\in\mathrm{ESp}_{\psi_n}(R)$. So, $\mathrm{ESp}_{2n}(R)\subseteq\mathrm{ESp}_{\psi_n}(R)$. 
	\qed
\end{proof}
	
\medskip
A relative version of the above lemma with respect to an ideal is as follows.

\begin{lemma}(Lemma 3.5, \cite{JOA})
\label{lemma:3.8}
	Let $R$ be a ring with $R=2R$ and $n\geq 2$. For an ideal $I$ of $R$, we have $\mathrm{ESp}_{\psi_n}(R,I)=\mathrm{ESp}_{2n}(R,I)$. 
\end{lemma}

\noindent
\textbf{Notation:} Let $\varphi$ be an invertible skew-symmetric matrix of size $2n$ over $R$. 
\begin{equation*}
	\mathrm{Sp}_{\varphi\otimes R[X]}(R[X]):=\{\alpha\in\mathrm{GL}_{2n}(R[X])\;\big|\; \alpha^t\varphi\alpha=\varphi\}.
\end{equation*}

By $\mathrm{ESp}_{\varphi\otimes R[X]}(R[X])$, we mean the elementary symplectic group generated by $C_\varphi(v)$ and $R_\varphi(v)$, where $v\in R[X]^{2n-1}$.\\

\medskip
Now, we will prove a result which gives a relation between the relative elementary symplectic group $\mathrm{ESp}_{2n}(R[X],(X))$ and the group $\mathrm{Sp}_{2n}(R[X],(X))$. The proof uses ideas from Lemma 2.7 of \cite{Kop}. 

\begin{lemma}
\label{ESp cap Sp X}
	For a ring $R$, we have
	\begin{equation*}
		\mathrm{ESp}_{2n}(R[X],(X))=\mathrm{ESp}_{2n}(R[X])\cap \mathrm{Sp}_{2n}(R[X],(X)).
	\end{equation*}
\end{lemma}
\begin{proof}
	Elements of $\mathrm{ESp}_{2n}(R[X],(X))$ are generated by elements of the form $\gamma=se_{kl}(v(X))se_{ij}(u(X))se_{kl}(-v(X))$, where $v(X)\in R[X]$ and $u(X)\in (X)$. Note that $\gamma\in \mathrm{ESp}_{2n}(R[X])$ and $\gamma (mod\text{ }(X))=I_{2n}$. Therefore, $\mathrm{ESp}_{2n}(R[X],(X))\subseteq\mathrm{ESp}_{2n}(R[X])\cap \mathrm{Sp}_{2n}(R[X],(X))$.
	
	Let $\gamma\in \mathrm{ESp}_{2n}(R[X])\cap \mathrm{Sp}_{2n}(R[X],(X))$. Then, $\gamma=\prod_{k=1}^t se_{i_kj_k}(v_k(X))$, with $v_k(X)\in R[X]$ and $\prod_{k=1}^t se_{i_kj_k}(v_k(X) (mod\text{ }(X)))=I_{2n}$. Note that for each $k$, there exist $u_k\in R$ and $w_k(X)\in R[X]$ such that $v_k(X)=u_k+Xw_k(X)$. Hence, $\gamma=\prod_{k=1}^tse_{i_kj_k}(u_k+Xw_k(X))$ with $\prod_{k=1}^tse_{i_kj_k}(u_k)=I_{2n}$. By the splitting property of elementary generators ($se_{ij}(a+b)=se_{ij}(a)se_{ij}(b)$), we have 
	\begin{eqnarray*}
		&\gamma&=\prod_{k=1}^tse_{i_kj_k}(u_k)se_{i_kj_k}(Xw_k(X)) \\
		&&=\bigg(\prod_{k=1}^t\gamma_kse_{i_kj_k}(Xw_k(X))\gamma_k^{-1}\bigg)\prod_{k=1}^tse_{i_kj_k}(u_k) \\
		&&=\prod_{k=1}^t\gamma_kse_{i_kj_k}(Xw_k(X))\gamma_k^{-1}.\\
	\end{eqnarray*}
    Here, $\gamma_k=\prod_{s=1}^kse_{i_sj_s}(u_s)\in\mathrm{ESp}_{2n}(R)$ for $1\leq k\leq t$ and $se_{i_kj_k}(Xw_k(X))\in\mathrm{ESp}_{2n}((X))$. This shows that $\gamma\in \mathrm{ESp}_{2n}(R[X],(X))$. Therefore, $\mathrm{ESp}_{2n}(R[X],(X))=\mathrm{ESp}_{2n}(R[X])\cap \mathrm{Sp}_{2n}(R[X],(X))$.  
\qed
\end{proof}

\begin{corollary}
\label{ESp cap Sp}
	Let $R$ be a ring and I be a principal ideal of $R$. Then,
	\begin{equation*}
		\mathrm{ESp}_{2n}(R,I)=\mathrm{ESp}_{2n}(R)\cap \mathrm{Sp}_{2n}(R,I)
	\end{equation*} 
\end{corollary}
\begin{proof}
	Suppose $I=(a)$ for some $a\in R$. The result follows by substituting $X=a$ in Lemma \ref{ESp cap Sp X}.
\qed
\end{proof}


\section{Equality}
In this section, we give a relation between a skew-symmetric matrix of Pfaffian $1$ of size $2n$ and the standard skew-symmetric matrix over a euclidean domain (Lemma \ref{phi psi}) using which we prove the main result Theorem \ref{equality} and its relative version Theorem \ref{rel equality}.

\begin{lemma}(Lemma 3.5, \cite{Kop})
\label{Kopeiko equality}
If $R$ is a euclidean domain, then $\mathrm{Sp}_{2n}(R)=\mathrm{ESp}_{2n}(R)$.
\end{lemma}

\begin{lemma}(Proposition 5.4, \cite{Lam})
\label{um e1}
Let $R$ be a euclidean domain and $n\geq 2$. Then for any $a\in\mathrm{Um}_n(R)$, there exists $\beta\in\mathrm{E}_n(R)$ such that $a^t\beta=e_1^t$.
\end{lemma}

Now, we will prove a result which gives the relation between any skew-symmetric matrix of size $2n$ of Pfaffian $1$ and the standard skew-symmetric matrix $\psi_n$ over a euclidean domain. The proof uses ideas from Lemma 5.2 of \cite{JPAA}.

\begin{lemma}
\label{phi psi}
Let $R$ be a euclidean domain and $\varphi$ be a skew-symmetric matrix of size $2n$ of Pfaffian $1$ over $R$. Then $\varphi=(1\perp\epsilon)^t\psi_n(1\perp\epsilon)$ for some $\epsilon\in\mathrm{E}_{2n-1}(R)$.
\end{lemma}
\begin{proof}
	We will prove the result by induction on $n$. For $n=1$, the only skew-symmetric matrix of size $2$ of Pfaffian $1$ is $\psi_1$ and hence $\varphi=\psi_1$. In this case, we can take $\epsilon=1$.
	
	Assume that the result is true for any skew-symmetric matrix of size $2(n-1)$, that is, if $\varphi^\ast$ is a skew-symmetric matrix of size $2(n-1)$ of Pfaffian $1$ over $R$, then there exists $\gamma\in\mathrm{E}_{2n-3}(R)$ such that $\varphi^\ast=(1\perp\gamma)^t\psi_{n-1}(1\perp\gamma)$.
	
	We will prove the result for skew-symmetric matrix $\varphi$ of Pfaffian $1$ of size $2n$. Let $\varphi=\begin{pmatrix}0 & c\\-c^t & \nu\end{pmatrix}$. Then $c\in\mathrm{Um}_{2n-1}(R)$ and $\nu$ is a skew-symmetric matrix of size $2n-1$. By lemma \ref{um e1}, there exists $\beta\in\mathrm{E}_{2n-1}(R)$ such that $c^t\beta=e_1^t$. Then,
	\begin{equation*}
		(1\perp\beta)^t\varphi(1\perp\beta)=\begin{pmatrix}0 & e_1^t\\-e_1 & \mu	\end{pmatrix},
	\end{equation*}
    where $\mu=\beta^t\nu\beta$. Note that $\mu$ is skew-symmetric and hence can be written as $\begin{pmatrix} 0 & d^t\\-d & \varphi^\ast \end{pmatrix}$ for some $d\in R^{2n-2}$. Now,
    \begin{equation*}
    \begin{pmatrix}	1 & 0 & 0\\0 & 1 & -b(\varphi^\ast)^{-1}\\0 & 0 & I_{2n-2}
    \end{pmatrix}(1\perp\beta)^t\varphi(1\perp\beta)\begin{pmatrix} 	1 & 0 & 0\\0 & 1 & -b(\varphi^\ast)^{-1}\\0 & 0 & I_{2n-2} \end{pmatrix}^t= \begin{pmatrix}
    0 & 1 & 0\\-1 & 0 & 0\\0 & 0 & \varphi^\ast\end{pmatrix}.
    \end{equation*}
   Note that Pfaffian of $\varphi^\ast$ is $1$. By induction hypothesis, we have $\gamma\in\mathrm{E}_{2n-3}(R)$ with $\varphi^\ast=(1\perp\gamma)^t\psi_{n-1}(1\perp\gamma)$. Let
    \begin{equation*}
    (I_3\perp\gamma)^{-1}\begin{pmatrix}1 & 0 & 0\\0 & 1 & -b(\varphi^\ast)^{-1}\\0 & 0 & I_{2n-2}\end{pmatrix}(1\perp\beta)^t=(1\perp\epsilon^t)^{-1}.	
    \end{equation*}  
    Then, $\epsilon\in\mathrm{E}_{2n-1}(R)$ and $\varphi=(1\perp\epsilon)^t\psi_n(1\perp\epsilon)$. 
\qed
\end{proof}

\begin{thm}
\label{equality}
Suppose $R$ is a euclidean domain with $R=2R$. Let $\varphi$ is a skew-symmetric matrix of size $2n$ of Pfaffian 1 over $R$ with $n\geq 2$. Then $\mathrm{Sp}_{\varphi}(R)=\mathrm{ESp}_{\varphi}(R)$.
\end{thm}
\begin{proof}
	By Lemma \ref{phi psi}, there exists $\epsilon\in\mathrm{E}_{2n-1}(R)$ such that $\varphi=(1\perp\epsilon)^t\psi_n(1\perp\epsilon)$. By Lemma \ref{phi phi star} and Lemma \ref{psi 2n}, we have
	\begin{eqnarray*}
	\mathrm{Sp}_\varphi(R)&=(1\perp\epsilon)^{-1}\mathrm{Sp}_{2n}(R)(1\perp\epsilon),\\
	\mathrm{ESp}_{\varphi}(R)&=(1\perp\epsilon)^{-1}\mathrm{ESp}_{2n}(R)(1\perp\epsilon).
	\end{eqnarray*}
    By Lemma \ref{Kopeiko equality}, it follows that $\mathrm{Sp}_{\varphi}(R)=\mathrm{ESp}_{\varphi}(R)$. 
\qed
\end{proof}

\begin{thm}
	\label{rel equality}
	Let $R$ be a euclidean domain with $R=2R$ and $\varphi$ be a skew-symmetric matrix of size $2n$ of Pfaffian 1 over $R$. Let $n\geq 2$. Then for any ideal $I$ of $R$, we have $\mathrm{Sp}_{\varphi}(R,I)=\mathrm{ESp}_{\varphi}(R,I)$.
\end{thm}
\begin{proof}
	As $R$ is euclidean, the ideal $I$ is a principal ideal. Therefore, 
    \begin{eqnarray*}
    	&\mathrm{ESp}_{2n}(R,I)&=\mathrm{ESp}_{2n}(R)\cap \mathrm{Sp}_{2n}(R,I) \text{ [by Corollary \ref{ESp cap Sp}]}\\
    	&&=\mathrm{Sp}_{2n}(R)\cap \mathrm{Sp}_{2n}(R,I) \text{ [by Theorem \ref{Kopeiko equality}]}\\
    	&&=\mathrm{Sp}_{2n}(R,I).
    \end{eqnarray*}
    By Lemma \ref{phi psi}, there exists $\epsilon\in \mathrm{E}_{2n-1}(R)$ such that $\varphi=(1\perp\epsilon)^t\psi_n(1\perp\epsilon)$. By Lemma \ref{lemma:3.6} and Lemma \ref{lemma:3.8}, it follows that 
    \begin{eqnarray*}
    	&\mathrm{ESp}_{\varphi}(R,I)&=(1\perp\epsilon)^{-1}\mathrm{ESp}_{2n}(R,I)(1\perp\epsilon)\\
    	&&=(1\perp\epsilon)^{-1}\mathrm{Sp}_{2n}(R,I)(1\perp\epsilon).
    \end{eqnarray*}
    Note that $(1\perp\epsilon)^{-1}\mathrm{Sp}_{2n}(R,I)(1\perp\epsilon)=\mathrm{Sp}_\varphi(R,I)$. Therefore, we have $\mathrm{Sp}_{\varphi}(R,I)=\mathrm{ESp}_{\varphi}(R,I)$.
\qed
\end{proof}

\end{document}